\documentclass[12pt]{article}
\usepackage[english]{babel}%
\usepackage{amssymb,amsmath,amsthm,indentfirst}
\def\thalf{\textstyle{\frac{1}{2}}\displaystyle}
  \def\ttt{\textstyle{\frac{2}{3}}\displaystyle}

\def\ve{\varepsilon}    
\def\vk{\varkappa}
\def\la{\lambda}

            \def\vp{\varphi}
         
\def\nb{\nabla}
        \def\bB{{\bf B}}
\def\bb{{\bf b}}        
        
       \def\bF{{\bf F}}

\def\bq{{\bf q}}
\def\bu{{\bf u}}        \def\bv{{\bf v}}
\def\bw{{\bf w}}        
        
\def\bbe{\boldsymbol\beta}

\def\pa{\partial}

\def\wh{\widehat}

\def\be{\begin{equation}}       \def\ee{\end{equation}}
\def\barr{\begin{array}}        \def\earr{\end{array}}
\def\bp{\begin{proposition}}    \def\ep{\end{proposition}}
\def\bt{\begin{theorem}}        \def\et{\end{theorem}}
\def\ds{\displaystyle}     
\def\lb{\label}
\def\nn{\nonumber}
\def\lf{\left}    \def\rg{\right}
\def\leq{\le}     \def\geq{\ge}
\DeclareMathOperator{\dv}{div}
\def\leq{\leqslant}      \def\geq{\geqslant}
\newtheorem{proposition}{\indent Proposition}
\begin{document}
\begin{center}
{\LARGE{\bf On a regularization of the magnetic}}
\vskip 0.3cm
{\LARGE{\bf gas dynamics system of equations}}
\vskip 0.5cm
{\large Bernard Ducomet
\footnote{CEA, DAM, DIF, F--91297, Arpajon, France,
E-mail: \it{bernard.ducomet@cea.fr}}
{\large and Alexander Zlotnik}
\footnote{Department of Mathematics at Faculty of Economics,
National Research University Higher School of Economics,
Myasnitskaya 20, 101990 Moscow, Russia}
\footnote{Department of Mathematical Modelling,
National Research University Moscow Power Engineering Institute,
Krasnokazarmennaya 14, 111250 Moscow, {\sc Russia},
E-mail: \it{azlotnik2008@gmail.com}}}
\end{center}
\vskip 0.5cm
\begin{center}
{\it To the memory of S.D. Ustjugov.}
\end{center}
\vskip 0.5cm
\begin{abstract}
A brief derivation of a specific regularization for the magnetic gas dynamic system of equations is given in the case of general equations of gas state (in presence of a body force and a heat source). The entropy balance equation in two forms is also derived for the system.
For a constant regularization parameter and under a standard condition on the heat source, we show that the entropy production rate is nonnegative.
\end{abstract}
\large
\section{Introduction}
Regularized (or quasi-) gas dynamic systems of equations are exploited for various purposes including the construction a class of so-called kinetically consistent finite-difference methods for gas dynamics simulations. The corresponding background, designing of the finite-difference methods and various applications are presented in monographs \cite{Ch04}-\cite{Sh09}.
\par Several new issues in the mathematical treatment of this approach have been recently developed in \cite{Z06}-\cite{Z12MMa}. In particular,
the case of a general equations of gas state (in presence of a body force and a heat source) have been covered in \cite{Z10DANa}-\cite{ZG11} where the law of non-decreasing entropy, the Petrovskii parabolicity and the linearized stability of equilibrium solutions have been established.
\par In Magneto Gas Dynamics (MGD) \cite{LL,EM}, an application of this approach has been recently given.
In \cite{EU11a,EU11b} the corresponding regularized system has been formally written in the case of a perfect polytropic gas (in absence of body forces and heat sources)
and some successful numerical results has been presented in 1D and 2D cases aimed at astrophysical applications.
\par In this paper, following a recent formalism from \cite{Z12MMa}, we give a brief complete derivation of the regularized MGD system of equations in the standard form of mass, momentum and total energy balance equations together with the Faraday equation covering the case of general gas state equations (in presence of body force and heat source).
Our formulation of the equations is more standard and seems more suitable for further discretization.
In addition, we present a useful regularized internal energy balance equation.
\par
It is well known that a crucial point in the physical and mathematical (see \cite{FN} for a complete account of this last point) correctness of a gas dynamics system is an adequate entropy balance equation, and
our main result is the derivation of such an entropy balance equation for the regularized MGD system of equations. We write it down  in two forms enlarging the recent corresponding results from \cite{Z10DANa}, \cite{Z10MMb,ZG11}, moreover,
for a constant regularization parameter and under a standard condition on the heat source, we prove that the corresponding entropy production rate is nonnegative.
\section{A regularization of the magnetic gas dynamics system of equations}
\par We begin with the classical Navier-Stokes system of equations for a viscous compressible gas flow taking into account the magnetic field, a body force and a heat source. The system consists of the mass, the impulse and the total energy balance equation together with the
Faraday equation and the equation of the absence of magnetic charge
\begin{gather}
 \pa_t\rho+\dv(\rho\bu)=0,
\lb{eq1ns}\\[1mm]
 \pa_t(\rho\bu)+\dv(\rho\bu\otimes\bu-\bB\otimes\bB)+\nb\lf(p+\thalf|\bB|^2\rg)
 =\dv\Pi_{NS}+\rho\bF,
\lb{eq2ns}\\[1mm]
 \pa_t\lf(E+\thalf|\bB|^2\rg)+\dv\lf[(E+p+|\bB|^2)\bu-(\bu\cdot\bB)\bB\rg]
\nn\\[1mm]
 =\dv(\vk\nb\theta+\Pi_{NS}\bu)+\rho\bu\cdot\bF+Q,
\lb{eq3ns}\\[1mm]
\pa_t\bB+\dv(\bu\otimes\bB-\bB\otimes\bu)=0,
\lb{eq4nsb}\\[1mm]
\dv\bB=0.
\lb{eq4nsbb}
\end{gather}
\par We consider the gas density $\rho>0$, the velocity $\bu=(u_1,\dots,u_n)$, the absolute temperature $\theta>0$ and the magnetic field strength $\bB$ as the basic unknown functions (the multiplier $\frac{1}{\sqrt{4\pi}}$ is included into the definition of $\bB$).
In addition, the equations include the total non-magnetic energy $\ds{E=\frac{1}{2}\,\rho|\bu|^2+\rho\ve}$, the pressure $p$ and the specific internal energy $\ve$. The system is considered for $(x,t)\in\Omega\times (0,T)$, where $\Omega$ is a domain in $\mathbb{R}^n$, $n\geq 1$.
\par Concerning the notation, hereafter the operators $\dv$ and $\nb=(\pa_1,\dots,\pa_n)$ are taken with respect to the spatial variables $x=(x_1,\dots,x_n)$.
Also $\pa_i$ and $\pa_t$ are the partial derivatives in $x_i$ and $t$.
The divergence of a tensor is taken with respect to its first index.
The signs
$\otimes$ and $\cdot$ denote the tensor and inner products of vectors, and in the inner products such as $\bu\cdot\nb\vp$ the sign $\cdot$ is omitted for brevity.
Also the sign $:$ means the inner product of tensors.
\par We take general state equations
\be
 p=p(\rho,\theta),\ \ \ve=\ve(\rho,\theta)
\lb{eqst}
\ee
linked by the Maxwell relation
\be
 p=\theta p_\theta+\rho^2\ve_\rho
\lb{eq4nsm}
\ee
and satisfying the thermodynamic stability conditions
\[
 p_\rho\geq 0,\ \ \ve_\theta>0.
\]
Hereafter $p_\rho$, $p_\theta$, $\ve_\rho$ and $\ve_\theta$ are partial derivatives of the state functions \eqref{eqst}.
\par In the above equations, $\Pi_{NS}$ is the classical Navier-Stokes viscous stress tensor
\[
 \Pi_{NS}=\Pi_{NS}(\bu)=\mu\lf[2\mathbb{D}(\bu)-\ttt\,(\dv\bu)\mathbb{I}\rg]
 +\la(\dv\bu)\mathbb{I},\
 \mathbb{D}(\bu)=\thalf\,(\nabla\bu+\nabla\bu^T)
\]
with the dynamic viscosity coefficient $\mu=\mu(\rho,\theta)\geq 0$ and the bulk viscosity coefficient $\lambda=\lambda(\rho,\theta)\geq 0$,  also  $\mathbb{I}$ is the identity tensor (of order $n$) and $\nabla\bu=\{\pa_iu_j\}_{i,j=1}^n$.
Moreover, $\vk=\vk(\rho,\theta)\geq 0$ is the heat conductivity coefficient,
the given functions $\bF=\bF(x,t)$ and  $Q=Q(x,t)\geq 0$ are the density of body forces and the power of heat sources.
The magnetic viscosity is neglected.
\par To regularize this system of equations, in general we follow \cite{E07,Sh09} (see also \cite{E11}) but apply the very recent simpler formalism from \cite{Z12MMa} and replace terms
\begin{gather*}
 \rho\bu,\ \
 \rho\bu\otimes\bu-\bB\otimes\bB,\ \
 \nb\lf(p+\thalf|\bB|^2\rg)-\rho\bF,
\\[1mm]
 (E+p+|\bB|^2)\bu-(\bu\cdot\bB)\bB,\ \
 \rho\bu\cdot\bF,\ \
 \bu\otimes\bB-\bB\otimes\bu
\end{gather*}
in the divergent summands in equations \eqref{eq1ns}, \eqref{eq2ns}, \eqref{eq3ns} and \eqref{eq4nsb} respectively by
\begin{gather*}
\rho\bu+\tau\hat{\pa_t}(\rho\bu),\ \ \rho\bu\otimes\bu-\bB\otimes\bB+\tau\hat{\pa_t}(\rho\bu\otimes\bu-\bB\otimes\bB),
\\[1mm]
 \nb[p+\thalf|\bB|^2+\tau\hat{\pa_t}\lf(p+\thalf|\bB|^2\rg)] -(\rho+\tau\hat{\pa_t}\rho)\bF,
\\[1mm]
 (E+p+|\bB|^2)\bu-(\bu\cdot\bB)\bB
 +\tau\hat{\pa_t}[(E+p+|\bB|^2)\bu-(\bu\cdot\bB)\bB],
\\[1mm]
 [\rho\bu+\tau\hat{\pa_t}(\rho\bu)]\cdot\bF,\ \
 \bu\otimes\bB-\bB\otimes\bu
 +\tau\hat{\pa_t}(\bu\otimes\bB-\bB\otimes\bu)
\end{gather*}
with a relaxation parameter $\tau=\tau(\rho,\ve,\bu,\bB)>0$.
Here the hat $\hat{\cdot}$ over the derivative $\pa_t$ means that it is calculated by virtue of the equations neglecting viscosity and heat conductivity (i.e., for zero $\Pi_{NS}$ and $\vk$) as for the Euler MGD system of equations.
\begin{proposition}
The regularized MGD system of equations has the form
\begin{gather}
 \pa_t\rho+\dv[\rho(\bu-\bw)]=0,
\lb{eq1h}\\[2mm]
 \pa_t(\rho\bu)+\dv[\rho(\bu-\bw)\otimes\bu-\bB\otimes\bB]+\nb\lf(p+\thalf|\bB|^2\rg)
\nn\\[1mm]
 =\dv\Pi+[\rho-\tau\dv(\rho\bu)]\bF,
\lb{eq2h}\\[2mm]
 \pa_t\lf(E+\thalf|\bB|^2\rg)+\dv\lf[(E+p)(\bu-\bw)+|\bB|^2(\bu-\wh{\bw})
 -((\bu-\wh{\bw})\cdot\bB)\bB\rg]
\nn\\[1mm]
 =\dv\lf[-\bq+\tau(\bbe\cdot\bB)\bu+\Pi\bu\rg]+\rho(\bu-\bw)\cdot\bF+Q,
\lb{eq3h}\\[2mm]
\pa_t\bB+\dv[(\bu-\wh{\bw})\otimes\bB-\bB\otimes(\bu-\wh{\bw})]
=\dv[\tau(\bu\otimes\bbe-\bbe\otimes\bu)],
\lb{eq4hb}\\[1mm]
\dv\bB=0
\lb{eq4hbb}
\end{gather}
provided that $\dv\bB|_{t=0}=0$.
\par Here the auxiliary velocity vector-functions $\bw$ and $\wh{\bw}$ are given by formulas
\begin{gather}
 \bw=\frac{\tau}{\rho}\,[\dv(\rho\bu\otimes\bu-\bB\otimes\bB)
 +\nb\lf(p+\thalf|\bB|^2\rg)-\rho\bF],
\lb{pr100}\\[1mm]
 \wh{\bw}=\frac{\tau}{\rho}\,[\rho(\bu\nb)\bu-\dv(\bB\otimes\bB)
 +\nb\lf(p+\thalf|\bB|^2\rg)-\rho\bF].
\lb{pr10}
\end{gather}
The non-symmetric regularized viscous stress tensor has the form
\begin{gather}
 \Pi=\Pi_{NS}+\rho\bu\otimes\wh{\bw}-\tau(\bbe\otimes\bB+\bB\otimes\bbe)
\nn\\[1mm]
 +\tau\lf(\bu\nb p
 +\rho C_s^2\dv\bu+\bbe\cdot\bB-\frac{p_\theta}{\rho\ve_\theta}Q\rg)\mathbb{I},
\lb{eq4h}
\end{gather}
where $C_s\geq 0$ is the speed of sound in the gas defined by a known formula
\be
 C_s^2=p_\rho+\frac{\theta p_\theta^2}{\rho^2\ve_\theta}
\lb{eq5ns}
\ee
(for example see \cite{K02}).
The regularized heat flux $\bq$ is given by a known formula
\be
 -\bq=\vk\nb\theta
 +\tau\lf[\rho\lf(\bu\nb\ve-\frac{p}{\rho^2}\,\bu\nb\rho\rg)-Q
 \rg]\bu,
\lb{eq5h}
\ee
and the auxiliary vector-function $\bbe$ has the form
\be
 \bbe=\dv(\bu\otimes\bB-\bB\otimes\bu).
\lb{eq6hb}
\ee
\end{proposition}
\begin{proof}
By virtue of equations \eqref{eq1ns} and \eqref{eq2ns} we have
\begin{gather}
\hat{\pa_t}\rho=-\dv(\rho\bu),\ \
\label{eqE3a}\\[1mm]
\tau\hat{\pa_t}(\rho\bu)=-\rho\bw
\label{eqE3}
\end{gather}
with $\bw$ given by \eqref{pr100}.
The latter formula implies the regularized mass balance equation \eqref{eq1h}.
Also the following formula
\begin{gather}
\tau\hat{\pa_t}(\rho\bu\otimes\bu)
 =\tau\{\hat{\pa_t}(\rho\bu)\otimes\bu
 +\bu\otimes[\hat{\pa_t}(\rho\bu)-(\hat{\pa_t}\rho)\bu]\}
\nn\\[1mm]
=-\rho\bw\otimes\bu-\bu\otimes[\rho\bw-\tau
\dv(\rho\bu)\bu]
 =-\rho\bw\otimes\bu-\bu\otimes\rho\wh{\bw}
\label{eqE4}
\end{gather}
holds taking into account that
\begin{gather}
 \rho\bw=\rho\wh{\bw}+\tau\dv(\rho\bu)\bu
\label{eqE4a}
\end{gather}
with $\wh{\bw}$ given by \eqref{pr10}. Notice also the related useful formula
\begin{gather}
 \tau\hat{\pa_t}\bu
 =\frac{\tau}{\rho}\,\hat{\pa_t}(\rho\bu)-\frac{\tau}{\rho}\,(\hat{\pa_t}\rho)\bu
 =-\wh{\bw},
\label{eqE4b}
\end{gather}
see \eqref{eqE3}  and \eqref{eqE3a}.
\par The well known equation
\be
 \hat{\pa_t}p=-\lf(\bu\nabla p+\rho C_s^2\dv\bu-\frac{p_\theta}{\rho\ve_\theta}\,Q\rg)
\lb{eq4ns}
\ee
holds, where the speed of sound $C_s$ is given by \eqref{eq5ns}.
The right-hand side of the equation does not depend on $\bB$.
\par Since clearly
\begin{gather}
 \hat{\pa_t}\bB=-\bbe
\lb{eq4nsa}
\end{gather}
with $\bbe$ given by \eqref{eq6hb},
we get
\begin{gather}
 \hat{\pa_t}(\bB\otimes\bB)=-(\bbe\otimes\bB+\bB\otimes\bbe),\ \
 \hat{\pa_t}\lf(\thalf|\bB|^2\rg)=-\bbe\cdot\bB,
\lb{eq4nsc}
\end{gather}
and after recalling \eqref{eqE4b}
\begin{gather}
 \tau\hat{\pa_t}(\bu\otimes\bB-\bB\otimes\bu)
 =-[\wh{\bw}\otimes\bB-\bB\otimes\wh{\bw}+\tau(\bu\otimes\bbe-\bbe\otimes\bu)].
\lb{eq4nsd}
\end{gather}
Collecting the above formulas \eqref{eqE4}, \eqref{eq4ns} and \eqref{eq4nsc},
we derive the regularized impulse balance equation \eqref{eq2h}
with the regularized viscosity tensor in the form \eqref{eq4h}.
Also formula \eqref{eq4nsd} implies the regularized Faraday equation \eqref{eq4hb}.
\par Obviously
\[
 \dv\dv(\textbf{a}\otimes\bb-\bb\otimes\textbf{a})=0.
\]
Thus equation \eqref{eq4hb} implies that $\pa_t\dv\bB=0$ and then $\dv\bB=0$ provided that $\dv\bB|_{t=0}=0$. Notice that due to this formula also the following property holds:
\be
 \dv\bbe=0.
\lb{pr11}
\ee
\par It remains to derive the regularized energy balance equation \eqref{eq3h} that is the most cumbersome point. Following \cite{Z12MMa},
we write down
\be
 \hat{\pa_t}\lf(E+\thalf|\bB|^2\rg)=-(e-Q),
\lb{eq6ns}
\ee
where
\[
 e=\dv\lf[\lf(E+p+|\bB|^2\rg)\bu
 -(\bu\cdot\bB)\bB\rg]-\rho\bu\cdot\bF.
\]
Then
\[
 \hat{\pa_t}[(E+p+|\bB|^2)\bu]
 =[-(e-Q)+\hat{\pa_t}p+\hat{\pa_t}\bB\cdot\bB]\bu
 +\lf(E+p+|\bB|^2\rg)\hat{\pa_t}\bu.
\]
By virtue of \eqref{eq4nsa}, \eqref{eqE4b} and \eqref{eqE4a} we find
\begin{gather*}
 \tau\hat{\pa_t}\lf[\lf(E+p+|\bB|^2\rg)\bu\rg]
 =-\tau\lf\{\dv\lf[\lf(E+p+|\bB|^2\rg)\bu\rg]
 -\rho\bu\cdot\bF\rg.
\\[1mm]
 \lf.-\dv[(\bu\cdot\bB)\bB]-\hat{\pa_t}p+\bbe\cdot\bB-Q\rg\}\bu
 -\lf(E+p+|\bB|^2\rg)\lf(\bw-\frac{\tau}{\rho}\,\dv(\rho\bu)\bu\rg)
\\[1mm]
 =-\lf(E+p+|\bB|^2\rg)\bw
 -\tau\lf\{\nb\frac{E+p+|\bB|^2}{\rho}\cdot\rho\bu-\rho\bF\cdot\bu\rg.
\\[1mm]
 \lf.-\dv[(\bu\cdot\bB)\bB]-\hat{\pa_t}p+\bbe\cdot\bB-Q\rg\}\bu.
\lb{pr13}
\end{gather*}
Furthermore
\[
 \nb\frac{E+p+|\bB|^2}{\rho}
 =\nb\ve+(\nb\bu)\bu+\frac1\rho\nb\lf(p+\thalf|\bB|^2\rg)
 +\frac1\rho\nb\lf(\thalf|\bB|^2\rg)-\frac{p+|\bB|^2}{\rho^2}\nb\rho.
\]
By virtue of \eqref{eqE4b} and \eqref{eq4nsa} we also have
\[
-\tau\hat{\pa_t}[(\bu\cdot\bB)\bB]
=(\wh{\bw}\cdot\bB)\bB+(\bu\cdot\bbe)\bB+(\bu\cdot\bB)\bbe.
\]
Consequently, also subtracting and adding the term $\tau\dv(\bB\otimes\bB)\cdot\bu$, rearranging the summands and recalling formula \eqref{pr10}, we obtain
\begin{gather}
 \tau\hat{\pa_t}\lf[\lf(E+p+|\bB|^2\rg)\bu-(\bu\cdot\bB)\bB\rg]
 =-\lf(E+p+|\bB|^2\rg)\bw
\nn\\[1mm]
 +(\wh{\bw}\cdot\bB)\bB-\tau M\bu
 -\tau\lf[\rho\lf(\bu\nb\ve-\frac{p}{\rho^2}\bu\nb\rho\rg)-Q\rg]\bu
\nn\\[1mm]
 -\lf\{(\rho\wh{\bw}\cdot\bu)\bu
 -\tau[(\bu\cdot\bbe)\bB+(\bu\cdot\bB)\bbe]
 +\tau(-\hat{\pa_t}p+\bbe\cdot\bB)\bu\rg\},
\lb{pr13bis}
\end{gather}
where
\[
 M=\dv(\bB\otimes\bB)\cdot\bu-\dv[(\bu\cdot\bB)\bB]
 +\bu\nb\lf(\thalf|\bB|^2\rg)
 -\frac{|\bB|^2}{\rho}\bu\nb\rho.
\]
\par Differentiating easily leads to formulas
\begin{gather}
 \dv(\bB\otimes\bB)\cdot\bu
 =\dv[(\bu\cdot\bB)\bB]-(\bB\nb)\bu\cdot\bB,
\nn\\[1mm]
 \bu\nb\lf(\thalf|\bB|^2\rg)=(\bu\nb)\bB\cdot\bB,
\nn\\[1mm]
 \bbe=(\dv\bu)\bB+(\bu\nb)\bB-(\bB\nb)\bu
\lb{pr14c}
\end{gather}
(the last one uses the property $\dv\bB=0$). Thus
\be
 M=\bbe\cdot\bB-(\dv\bu)|\bB|^2-\frac{|\bB|^2}{\rho}\bu\nb\rho
 =\bbe\cdot\bB-\frac{|\bB|^2}{\rho}\dv(\rho\bu).
\lb{pr15}
\ee
\par Applying the formula
\[
 (\textbf{a}\otimes\bb)\bu=(\bb\cdot\bu)\textbf{a},
\]
from \eqref{eq4h} we find that
\begin{gather}
 \Pi\bu-\Pi_{NS}\bu=(\rho\wh{\bw}\cdot\bu)\bu
 -\tau[(\bu\cdot\bB)\bbe+(\bu\cdot\bbe)\bB]
\nn\\[1mm]
 +\tau\lf(\bu\nb p
 +\rho C_s^2\dv\bu-\frac{p_\theta}{\rho\ve_\theta}Q+\bbe\cdot\bB\rg)\bu.
\lb{pr17}
\end{gather}
Recalling formula \eqref{eq4ns}, we recognize that this term represent one in the curly brackets in \eqref{pr13bis}.
\par Therefore inserting formula \eqref{pr15} (and recalling \eqref{eqE4a}) and \eqref{pr17} into \eqref{pr13bis}, we obtain
\begin{gather*}
\tau\hat{\pa_t}\lf[\lf(E+p+|\bB|^2\rg)\bu-(\bu\cdot\bB)\bB\rg]
 =-(E+p)\bw-|\bB|^2\wh{\bw}
\\[1mm]
 +(\wh{\bw}\cdot\bB)\bB+\bq+\vk\nb\theta-\tau(\bbe\cdot\bB)\bu-(\Pi\bu-\Pi_{NS}\bu),
\end{gather*}
where $\bq$ is given by formula \eqref{eq5h}. This together with \eqref{eqE3} straightforwardly leads to the regularized energy balance equation \eqref{eq3h}.
\end{proof}
\par The regularized MGD system of equations \eqref{eq1h}-\eqref{eq6hb}
generalizes to the magnetic situation the quasi-gas dynamics system in the case of real gas from \cite{Z10DANa}-\cite{ZG11}.
It also generalizes the quasi-MGD system from \cite{EU11b} to the case of real gases in presence of a body force and a heat source. Notice that our system is written in another more standard form that can be essential for further discretization.
\begin{proposition}
For the regularized MGD system of equations, the following internal energy balance equation holds
\begin{gather}
 \pa_t(\rho\ve)+\dv[\rho\ve(\bu-\bw)]+p\dv(\bu-\bw)
 =\dv[-\bq+\tau(\bu\cdot\bB)\bbe]
\nn\\[1mm]
 +\Pi:\nb\bu
 +\bw\nb p
 +\lf[-\dv(\bB\otimes\bB)+\nb\lf(\thalf|\bB|^2\rg)-\rho\bF\rg]\cdot\wh{\bw}
\nn\\[1mm]
 +\tau[(\bu\nb)\bB\cdot\bbe-(\bbe\nb)\bB\cdot\bu]
 +Q.
\lb{pr21}
\end{gather}
\end{proposition}
\begin{proof}
For the regularized MGD system of equations,
we subtract the impulse balance equation \eqref{eq2h} multiplied innerly by $\bu$ and the Faraday equation \eqref{eq4hb} multiplied innerly by $\bB$ from the energy balance equation \eqref{eq3h}. Since
\[
 \pa_t(\rho\bu)\cdot\bu+\dv[\rho(\bu-\bw)\otimes\bu]
 =\pa_t\lf(\thalf\rho|\bu|^2\rg)+\dv\lf[\thalf\rho|\bu|^2(\bu-\bw)\rg]
\]
taking into account the mass balance equation \eqref{eq1h}, we get
\begin{gather}
 \pa_t(\rho\ve)+\dv[(\rho\ve+p)(\bu-\bw)+|\bB|^2(\bu-\wh{\bw})
 -((\bu-\wh{\bw})\cdot\bB)\bB]
\nn\\[1mm]
 +\dv(\bB\otimes\bB)\cdot\bu-\bu\nb\lf(p+\thalf|\bB|^2\rg)
 -\dv[(\bu-\wh{\bw})\otimes\bB-\bB\otimes(\bu-\wh{\bw})]\cdot\bB
\nn\\[1mm]
 =\dv\lf[-\bq+\tau(\bbe\cdot\bB)\bu+\Pi\bu\rg]-\dv\Pi\cdot\bu
 -\dv[\tau(\bu\otimes\bbe-\bbe\otimes\bu)]\cdot\bB
\nn\\[1mm]
 -[\rho\bw-\tau\dv(\rho\bu)\bu]\cdot\bF+Q.
\lb{pr23}
\end{gather}
\par The following formulas are valid
\begin{gather*}
 \dv(\bv\otimes\bB-\bB\otimes\bv)\cdot\bB
\\[1mm]
 =\dv\lf[|\bB|^2\bv-(\bv\cdot\bB)\bB\rg]-(\bv\nb\bB)\cdot\bB+(\bB\nb)\bB\cdot\bv
\\[1mm]
 =\dv\lf[|\bB|^2\bv-(\bv\cdot\bB)\bB\rg]-\bv\nb\lf(\thalf|\bB|^2\rg)
 +\dv(\bB\otimes\bB)\cdot\bv
\end{gather*}
since $\dv(\bB\otimes\bB)=(\bB\nb)\bB$ due to $\dv\bB=0$ and
\begin{gather*}
 \dv[\tau(\bu\otimes\bbe-\bbe\otimes\bu)]\cdot\bB
\\[1mm]
 =\dv[\tau(\bbe\cdot\bB)\bu-\tau(\bu\cdot\bB)\bbe]
 -\tau[(\bu\nb)\bB\cdot\bbe-(\bbe\nb)\bB\cdot\bu].
\end{gather*}
\par Exploiting them in \eqref{pr23} for $\bv=\bu-\wh{\bw}$, differentiating in the terms $\dv[p(\bu-\bw)]$ and $\dv(\Pi\bu)$ and applying \eqref{eqE4a}, we derive
\begin{gather*}
 \pa_t(\rho\ve)+\dv[\rho\ve(\bu-\bw)]+p\dv(\bu-\bw)-\bw\nb p
\\[1mm]
 =\dv[-\bq+\tau(\bu\cdot\bB)\bbe]+\Pi:\nb\bu
  +\lf[-\dv(\bB\otimes\bB)+\nb\lf(\thalf|\bB|^2\rg)\rg]\cdot\wh{\bw}
\\[1mm]
 +\tau[(\bu\nb)\bB\cdot\bbe-(\bbe\nb)\bB\cdot\bu]
 -\rho\wh{\bw}\cdot\bF+Q.
\end{gather*}
This equality implies the internal energy balance equation \eqref{pr21}.
\end{proof}
\par Notice that the right-hand side of equation \eqref{pr21} does not depend on $\bB$ for $\tau=0$, i.e. without the regularization (and that is the reason why the right-hand side of equation \eqref{eq4ns} does not depend on $\bB$ too).
\par The crucial point of the physical correctness of a gas dynamics system is an adequate entropy balance equation.
The entropy $s=s(\rho,\ve)$ can be introduced by the Gibbs formulas
\be
 s_\rho=-\frac{p}{\rho^2\theta},\ \
 s_\ve=\frac{1}{\theta},
\lb{eq6}
\ee
see \cite{K02}.
The next proposition generalizes to the magnetic situation the corresponding results from \cite{Z10DANa,Z10MMb,ZG11}.
\begin{proposition}
For the regularized MGD system of equations, the following entropy balance equation
holds
\begin{gather}
 \pa_t(\rho s)+\dv[\rho s(\bu-\bw)]
 =\dv\lf(-\frac{\bq}{\theta}\rg)
 +\frac{1}{\theta}\,\Xi,
\lb{pr300}
\end{gather}
where the entropy production $\ds{\frac{1}{\theta}\,\Xi}$ is expressed by a formula
\begin{gather}
\ds{\Xi=\Xi_{NS,\,0}
  +\frac{\rho}{\tau}\,|\wh{\bw}|^2
  +\frac{\tau p_\rho}{\rho}\,[\dv(\rho\bu)]^2}
  +\ds{\frac{\tau\rho\ve_\theta}{\theta}
  \lf(\frac{\theta p_\theta}{\rho\ve_\theta}\dv\bu
  +\bu\nb\theta
  -\frac{Q}{2\rho\ve_\theta}\rg)^2}
\nn\\[1mm]
  \ds{+\tau|\bbe|^2+(\bB\cdot\bu)\bbe\nb\tau
  +Q\lf(1-\frac{\tau Q}
  {4\rho\theta\ve_\theta}\rg)}
\lb{f000a}
\end{gather}
with $\ds{\frac{1}{\theta}\,\Xi_{NS,\,0}}$ being the Navier-Stokes entropy production for $Q=0$:
\begin{gather}
\ds{\Xi_{NS,\,0}=2\mu\mathbb{D}_{ij}\mathbb{D}_{ij}
 +\lf(\la-\frac{2}{3}\,\mu\rg)(\dv\bu)^2
 +\frac{\vk}{\theta}\,|\nb\theta|^2\geq 0\ \ \text{for}\ n=1,2,3.}
\lb{eq8a}
\end{gather}
\par The entropy production can be also expressed by a formula
\begin{gather}
\ds{\Xi=\Xi_{NS,\,0}
  +\frac{\rho}{\tau}\,|\wh{\bw}|^2
  +\frac{\tau}{\rho C_s^2}\lf(\rho C_s^2\dv\bu
  +\bu\nb p-\frac{p_\theta Q}{2\rho\ve_\theta}\rg)^2}
\nn\\[1mm]
\ds{+\frac{\tau\rho\theta}{c_p}\,\lf(\bu\nb s
  -\frac{Q}{2\rho\theta}\rg)^2}
  +\tau|\bbe|^2+(\bB\cdot\bu)\bbe\nb\tau
  +Q\lf(1-\frac{\tau Q}
  {4\rho\theta\ve_\theta}\rg)
\lb{f000b}
\end{gather}
provided that $p_\rho>0$. Here $c_p$ and $c_v=\ve_\theta$ are the specific heats of gas at constant pressure and at constant volume related by a formula \cite{K02}
\[
 \frac{C_s^2}{p_\rho}=\frac{c_p}{c_v}.
\]
\par Under conditions
\be
 \tau=\textrm{const},\ \ \frac{\tau Q}{4\rho\theta\ve_\theta}\leq 1,
\lb{f000}
\ee
the entropy production is nonnegative:
$\ds{\frac{1}{\theta}\,\Xi\geq 0}$ for $n=1,2,3$.
\end{proposition}
\begin{proof}
Extending the argument of \cite{Z10MMb} (see also \cite{Z10DANa,ZG11}), we introduce the total time derivative
\[
 D_t\vp\equiv\pa_t(\rho\vp)+\dv[\rho\vp(\bu-\bw)]
 =\rho\pa_t\vp+\rho(\bu-\bw)\nb\vp,
\]
see the mass balance equation (\ref{eq1h}).
In a standard manner we write down
\begin{gather}
 D_ts=\frac{p}{\theta}D_t\frac{1}{\rho}+\frac1\theta D_t\ve
 =\frac{1}{\theta}\lf[p\dv(\bu-\bw)+D_t\ve\rg]
\nn\\[1mm]
 =\dv\lf[-\frac{\bq}{\theta}+\frac{\tau}{\theta}(\bB\cdot\bu)\bbe\rg]
 +\frac{1}{\theta}\,\Xi
\lb{pr30}
\end{gather}
according to the Gibbs formulas (\ref{eq6}) and the internal energy balance equation (\ref{pr21}), where
\begin{gather*}
 \Xi=\frac{1}{\theta}\lf[-\bq+\tau(\bB\cdot\bu)\bbe\rg]\nb\theta+\Pi:\nb\bu
 +\bw\nb p
\nn\\[1mm] +\lf[-\dv(\bB\otimes\bB)+\nb\lf(\thalf|\bB|^2\rg)-\rho\bF\rg]\cdot\wh{\bw}
 +\tau[(\bu\nb)\bB\cdot\bbe-(\bbe\nb)\bB\cdot\bu]
 +Q.
\end{gather*}
\par Furthermore, we have
\begin{gather*}
 \Pi:\nb\bu=\Pi_{NS}:\nb\bu+\rho(\bu\nb)\bu\cdot\wh{\bw}
 +\tau\lf(\bu\nb p +\rho C_s^2\dv\bu-\frac{p_\theta}{\rho\ve_\theta}Q\rg)\dv\bu
\\[1mm]
 +\tau\lf[-(\bbe\nb)\bu\cdot\bB-(\bB\nb)\bu\cdot\bbe+(\bbe\cdot\bB)\dv\bu\rg]
\end{gather*}
and
\[
 \frac{1}{\theta}\vk\nb\theta\cdot\nb\theta+\Pi_{NS}:\nb\bu=\Xi_{NS,\,0}.
\]
Next clearly
\begin{gather}
 \rho(\bu\nb)\bu\cdot\wh{\bw}+\bw\nb p
 +\lf[-\dv(\bB\otimes\bB)+\nb\lf(\thalf|\bB|^2\rg)-\rho\bF\rg]\cdot\wh{\bw}
\nn\\[1mm]
 =\frac{\rho}{\tau}|\wh{\bw}|^2+\frac{\tau}{\rho}\dv(\rho\bu)\bu\nb p,
\lb{pr31}
\end{gather}
see formulas \eqref{pr10} and \eqref{eqE4a}. From \cite{Z10MMb,ZG11}
the following representation is known:
\begin{gather}
 \frac{1}{\theta}\tau\lf[\rho\lf(\bu\nb\ve-\frac{p}{\rho^2}\,\bu\nb\rho\rg)-Q
  \rg]\bu\nb\theta
 +\frac{\tau}{\rho}\dv(\rho\bu)\bu\nb p
\nn\\[1mm]
 +\tau\lf(\bu\nb p +\rho C_s^2\dv\bu-\frac{p_\theta}{\rho\ve_\theta}Q\rg)\dv\bu
 +Q
\nn\\[1mm]
 =\frac{\tau p_\rho}{\rho}\,[\dv(\rho\bu)]^2
  +\frac{\tau\rho\ve_\theta}{\theta}
  \lf(\frac{\theta p_\theta}{\rho\ve_\theta}\dv\bu
  +\bu\nb\theta
  -\frac{Q}{2\rho\ve_\theta}\rg)^2
  +Q\lf(1-\frac{\tau Q}{4\rho\theta\ve_\theta}\rg).
\lb{pr33}
\end{gather}
\par It remains to collect all the other magnetic terms (containing $\bB$) with the multiplier $\tau$. We have
\begin{gather*}
 \mathcal{M}=\frac{1}{\theta}(\bB\cdot\bu)\bbe\nb\theta
 -(\bbe\nb)\bu\cdot\bB
 -(\bB\nb)\bu\cdot\bbe+(\bbe\cdot\bB)\dv\bu
\\[1mm]
 +(\bu\nb)\bB\cdot\bbe-(\bbe\nb)\bB\cdot\bu
\\[1mm]
 =\frac{1}{\theta}(\bB\cdot\bu)\bbe\nb\theta
 +[-(\bB\nb)\bu+(\dv\bu)\bB+(\bu\nb)\bB]\cdot\bbe-(\bbe\nb)(\bB\cdot\bu)
\\[1mm]
 =\frac{1}{\theta}(\bB\cdot\bu)\bbe\nb\theta
 +|\bbe|^2-\dv[(\bB\cdot\bu)\bbe]
\\[1mm]
 =|\bbe|^2-\theta\dv\frac{(\bB\cdot\bu)\bbe}{\theta},
\end{gather*}
where formula \eqref{pr14c} and property \eqref{pr11} have been applied.
Consequently
\begin{gather}
 \dv\lf[\frac{\tau}{\theta}(\bB\cdot\bu)\bbe\rg]+\frac{\tau}{\theta}\mathcal{M}
 =\frac{1}{\theta}\lf(\tau|\bbe|^2+(\bB\cdot\bu)\bbe\nb\tau\rg).
\lb{pr35}
\end{gather}
Formulas \eqref{pr30}-\eqref{pr35} imply the stated ones \eqref{pr300} and \eqref{f000a}.
\par Another representation \eqref{eq8a} for $\Xi$ also follows from \cite{Z10DANa}, \cite{Z10MMb,ZG11}.
\end{proof}
\par Notice that formulas \eqref{f000a} and \eqref{f000b} remain valid in the case
$\tau\geq 0$ provided that one rewrites the term $\frac{\rho}{\tau}\,|\wh{\bw}|^2$ as
\[
 \frac{\tau}{\rho}\,|\rho(\bu\nb)\bu-\dv(\bB\otimes\bB)
 +\nb\lf(p+\thalf|\bB|^2\rg)-\rho\bF|^2.
\]
\smallskip\par
\textbf{Acknowledgments}
\smallskip\par
The paper has been initiated during the visit
of A. Zlotnik in summer 2011 to the the D\'epar\-te\-ment de Physique
Th\'eorique et Appliqu\'ee, CEA/DAM/DIF Ile de France (Arpajon), which he thanks for hospitality.
The study is carried out by him within The National Research University Higher School of Economics' Academic Fund Program in 2012-2013, research grant No. 11-01-0051 and also under financial support of the Russian Foundation for Basic Research, project 10-01-00136.


\begin{thebibliography}{99}
\bibitem{Ch04}
B.N. Chetverushkin.
Kinetic schemes and quasi-gas dynamic system of equations. CIMNE, Barcelona, 2008.
\bibitem{E07}
T.G. Elizarova.
Quasi-gas dynamic equations. Springer, Dordrecht, 2009.
\bibitem{Sh09}
Yu.V. Sheretov.
Continuum dynamics under spatiotemporal averaging.
RKhD, Moscow-Izhevsk, 2009 (in Russian).
\bibitem{Z06}
\textsl{}A.A. Zlotnik,
Classification of some modifications of the Euler system of equations,
Doklady Math. 73 (2) (2006) 302-306.
\bibitem{ZCh08}
A.A. Zlotnik and B.N. Chetverushkin,
Parabolicity of the quasi-gasdynamic system of equations,
its hyperbolic second--order modification, and the stability of small perturbations for them,
Comp. Maths. Math. Phys. 48 (3) (2008) 420-446.
\bibitem{Z10CMMP}
A.A. Zlotnik,
Energy equalities and estimates for barotropic quasi-gasdynamic and quasi-hydrodynamic systems of equations,
Comp. Maths. Math. Phys. 50 (2) (2010) 310-321.
\bibitem{Z10DANa}
A.A. Zlotnik,
Quasi-gasdynamic system of equations with general equations of state,
Doklady Math. 81 (2) (2010) 312-316.
\bibitem{Z10DANb}
A.A. Zlotnik,
Linearized stability of equilibrium solutions to the quasi-gas\-dy\-na\-mic system of equations, Doklady Math., 82 (2) (2010) 811-815.
\bibitem{Z10MMb}
A.A. Zlotnik,
On the quasi-gasdynamic system of equations with general equations of state and a heat source,
Math. Modelling 22 (7) (2010) 53-64 (in Russian).
\bibitem{ZG11}
A. Zlotnik and V. Gavrilin,
On quasi-gasdynamic system of equations with general equations of state and its
application.
Math. Modelling Anal. 16 (4) (2011) 509-526.
\bibitem{Z12MMa}
A.A. Zlotnik,
On construction of quasi-gasdynamic systems of equations and the barotropic system with the potential body force,
Math. Modeling 24 (4) (2012) 65-79 (in Russian).
\bibitem{LL}
L.D. Landau and E.M. Lifshitz, Electrodynamics of continuous media. Pergamon Press, Oxford, 1960.
\bibitem{EM}
A.C. Eringen and G.A. Maugin,
Electrodynamics of continua, Vol. 2: Fluids and complex media.  Springer Verlag, New-York, Berlin, Heidelberg, 1990.
\bibitem{FN}
E.~Feireisl and A.~Novotn\'y,
 Singular limits in thermodynamics of viscous fluids. Birkhauser, Basel, 2009.
\bibitem{EU11a}
T.G. Elizarova and S.D. Ustjugov,
Quasi-gas dynamics algorithm to solve equations of magnetohydrodynamics.
One-dimensional case. Keldysh Inst. Appl. Math. Moscow, preprint No. 1, 2011 (in Russian).
\begin{verbatim}
http://www.keldysh.ru/papers/2011/source/prep2011_01.pdf
\end{verbatim}
\bibitem{EU11b}
T.G. Elizarova and S.D. Ustjugov,
Quasi-gas dynamics algorithm to solve equations of magnetohydrodynamics. Multidimensional case. Keldysh Inst. Appl. Math.Moscow, preprint No. 30, 2011 (in Russian).
\begin{verbatim}
 http://www.keldysh.ru/papers/2011/source/prep2011_30.pdf
\end{verbatim}
\bibitem{E11}
T.G. Elizarova, Time averaging as an approximate technique for constructing quasi-gasdynamic and quasi-hydrodynamic equations.
Comp. Maths Math. Phys. 51 (11) (2011) 1973-1982.
\bibitem{K02}
I.A. Kvasnikov, Thermodynamics and statistical physics, Vol. 1: Theory of equilibrium systems and thermodynamics. Editorial URSS, Moscow, 2002 (in Russian).

\end{thebibliography}
\end{document}